\documentclass[12pt]{article}
\usepackage{amsfonts}
\usepackage{amsmath,amscd,amsthm,a4,amssymb}
\usepackage{amsmath,amscd,amsthm,a4,amssymb}

\newtheorem{theorem}{Theorem}
{}
\newtheorem{corollary}{Corollary}
{}
\newtheorem{definition}{Definition}
{}
\newtheorem{remark}{Remark}
{}

{}
\theoremstyle{plain}
{}

{}

{}

{}

{}

{}

{}

{}

{}

{}

{}

{}

{}

{}
\newtheorem{proposition}{Proposition}
{}

{}

\begin{document}
\begin{center}
{\Large \bf{Generalized Bicomplex Numbers and Lie Groups \ }}
\end{center}
\centerline{\large S\i dd\i ka \"{O}zkald\i \ Karaku\c{s} $^{1}$, Ferdag KAHRAMAN AKSOYAK $^{2}${\footnotetext{
{ E-mail: siddika.karakus@bilecik.edu.tr (S. \"{O}zkald\i \ Karaku\c{s}) $^{1}$; $^{2}$ferda@erciyes.edu.tr(F. Kahraman Aksoyak )}} }}

\
\centerline{\it $^{1}$Bilecik \c{S}eyh Edebali University, Department of Mathematics,
Bilecik, Turkey}

\centerline{\it $^{2}$Erciyes University, Department of Mathematics,
Kayseri, Turkey}

\begin{abstract}
In this paper, we define the generalized bicomplex numbers and give some
algebraic properties of them. Also, we show that some hyperquadrics in $%
\mathbb{R}^{4}$ and $\mathbb{R}_{2}^{4}$ are Lie groups by using generalized
bicomplex number product and obtain Lie algebras of these Lie groups.
Morever, by using tensor product surfaces, we determine some special Lie
subgroups of these hyperquadrics.\
\end{abstract}

\begin{quote}\small
\end{quote}
\begin{quote}\small
\end{quote}

\section{Introduction}

In mathematics, a Lie group is a group which is also a differentiable
manifold with the property that the group operations are differentiable. To
establish group structure on the surface is quite difficult. Even if the
spheres that admit the structure of a Lie group are only the 0-sphere $S^{0}$
(real numbers with absolute value $1$), the circle $S^{1}$ (complex numbers
with absolute value $1$), the 3-sphere $S^{3}$ (the set of quaternions of
unit form) and $S^{7}.$ A manifold $M$ carrying n linearly independent
non-vanishing vector fields is called parallelisable and a Lie group is
parallelisable. For even $n>1$ $S^{n}$ is not a Lie group because it can not
be parallelisable as a differentiable manifold. Thus $S^{n}$ is
parallelisable if and only $n=0,1,3,7$.

\"{O}zkald\i \ and Yayl\i \  \cite{kar} showed that a hyperquadric $P$ in $%
\mathbb{R}^{4}$ is a Lie group by using bicomplex number product. They
determined some special subgroups of this Lie group $P,$ by using the tensor
product surfaces of Euclidean planar curves. Karaku\c{s} \"{O}. and Yayl\i \
\cite{kara} showed that a hyperquadric $Q$ in $\mathbb{R}_{2}^{4}$ is a Lie
group by using bicomplex number product. They changed the rule of tensor
product and they gave a new tensor product rule in $\mathbb{R}_{2}^{4}.$ By
means of the tensor product surfaces of a Lorentzian plane curve and a
Euclidean plane curve, they determined some special subgroups of this Lie
group $Q.$ In \cite{ak1} and \cite{ak2},by using curves and surfaces which
are obtained by homothetic motion, were obtained some special subgroups of
these Lie groups $P$ and $Q,$ respectively.

In this paper, we define the generalized bicomplex numbers and give some
algebraic properties of them. Also, we show that some hyperquadrics in $%
\mathbb{R}^{4}$ and $\mathbb{R}_{2}^{4}$ are Lie groups by using generalized
bicomplex number product and obtain Lie algebras of these Lie groups.
Morever, by means of tensor product surfaces, we determine some special Lie
subgroups of these hyperquadrics and obtain left invariant vector fields of
these tensor product surfaces which are Lie groups.

\section{Preliminaries}

Bicomplex number is defined by the basis $\left \{ 1,i,j,ij\right \} $ where
$i,j,ij$ satisfy $i^{2}=-1,$ $j^{2}=-1,$ $ij=ji.$ Thus any bicomplex number $%
x $ can be expressed as $x=x_{1}1+x_{2}i+x_{3}j+x_{4}ij$, $\forall
x_{1},x_{2},x_{3},x_{4}\in \mathbb{R}.$ We denote the set of bicomplex
numbers by $C_{2}.$ For any $x=x_{1}1+x_{2}i+x_{3}j+x_{4}ij$ and $%
y=y_{1}1+y_{2}i+y_{3}j+y_{4}ij$ in $C_{2}$ the bicomplex number addition is
defined as

\begin{equation*}
x+y=\left( x_{1}+y_{1}\right) +\left( x_{2}+y_{2}\right) i+\left(
x_{3}+y_{3}\right) j+\left( x_{4}+y_{4}\right) ij\text{.}
\end{equation*}%
The multiplication of a bicomplex number $x=x_{1}1+x_{2}i+x_{3}j+x_{4}ij$ by
a real scalar $\lambda $ is defined as
\begin{equation*}
\lambda x=\lambda x_{1}1+\lambda x_{2}i+\lambda x_{3}j+\lambda x_{4}ij\text{.%
}
\end{equation*}%
With this addition and scalar multiplication, $C_{2}$ is a real vector space.

Bicomplex number product, denoted by $\times $, over the set of bicomplex
numbers $C_{2}$\ is given by
\begin{eqnarray*}
x\times y &=&\left( x_{1}y_{1}-x_{2}y_{2}-x_{3}y_{3}+x_{4}y_{4}\right)
+\left( x_{1}y_{2}+x_{2}y_{1}-x_{3}y_{4}-x_{4}y_{3}\right) i \\
&&+\left( x_{1}y_{3}+x_{3}y_{1}-x_{2}y_{4}-x_{4}y_{2}\right) j+\left(
x_{1}y_{4}+x_{4}y_{1}+x_{2}y_{3}+x_{3}y_{2}\right) ij\text{.}
\end{eqnarray*}%
Vector space $C_{2}$ together with the bicomplex number product $\times $ is
a real algebra.

Since the bicomplex algebra is associative, it can be considered in terms of
matrices. Consider the set of matrices%
\begin{equation*}
Q=\left \{ \left(
\begin{array}{cccc}
x_{1} & -x_{2} & -x_{3} & x_{4} \\
x_{2} & x_{1} & -x_{4} & -x_{3} \\
x_{3} & -x_{4} & x_{1} & -x_{2} \\
x_{4} & x_{3} & x_{2} & x_{1}%
\end{array}%
\right) ;\text{ \  \  \  \  \  \ }x_{i}\in \mathbb{R}\text{ ,\  \  \  \ }1\leq i\leq
4\right \} \text{.}
\end{equation*}%
The set $Q$ together with matrix addition and scalar matrix multiplication
is a real vector space. Furthermore, the vector space together with matrix
product is an algebra.

The transformation
\begin{equation*}
g:C_{2}\rightarrow Q
\end{equation*}%
given by
\begin{equation*}
g\left( x=x_{1}1+x_{2}i+x_{3}j+x_{4}ij\right) =\left(
\begin{array}{cccc}
x_{1} & -x_{2} & -x_{3} & x_{4} \\
x_{2} & x_{1} & -x_{4} & -x_{3} \\
x_{3} & -x_{4} & x_{1} & -x_{2} \\
x_{4} & x_{3} & x_{2} & x_{1}%
\end{array}%
\right)
\end{equation*}%
is one to one and onto. Morever $\forall x,y\in C_{2}$ and $\lambda \in
\mathbb{R},$ we have
\begin{eqnarray*}
g\left( x+y\right) &=&g\left( x\right) +g\left( y\right) \\
g\left( \lambda x\right) &=&\lambda g\left( x\right) \\
g\left( xy\right) &=&g\left( x\right) g\left( y\right) \text{.}
\end{eqnarray*}%
Thus the algebras $C_{2}$ and $Q$ are isomorphic.

Let $x\in C_{2}.$ Then $x$ can be expressed as $x=\left( x_{1}+x_{2}i\right)
+\left( x_{3}+x_{4}i\right) j.$ In that case, there is three different
conjugations for bicomplex numbers as follows:%
\begin{eqnarray*}
x^{t_{i}} &=&\left[ \left( x_{1}+x_{2}i\right) +\left( x_{3}+x_{4}i\right) j%
\right] ^{t_{i}}=\left( x_{1}-x_{2}i\right) +\left( x_{3}-x_{4}i\right) j \\
x^{t_{j}} &=&\left[ \left( x_{1}+x_{2}i\right) +\left( x_{3}+x_{4}i\right) j%
\right] ^{t_{j}}=\left( x_{1}+x_{2}i\right) -\left( x_{3}+x_{4}i\right) j \\
x^{t_{ij}} &=&\left[ \left( x_{1}+x_{2}i\right) +\left( x_{3}+x_{4}i\right) j%
\right] ^{t_{ij}}=\left( x_{1}-x_{2}i\right) -\left( x_{3}-x_{4}i\right) j%
\text{.}
\end{eqnarray*}%
And we can write
\begin{eqnarray*}
x\times x^{t_{i}} &=&\left( x_{1}^{2}+x_{2}^{2}-x_{3}^{2}-x_{4}^{2}\right)
+2\left( x_{1}x_{3}+x_{2}x_{4}\right) j \\
x\times x^{t_{j}} &=&\left( x_{1}^{2}-x_{2}^{2}+x_{3}^{2}-x_{4}^{2}\right)
+2\left( x_{1}x_{2}+x_{3}x_{4}\right) i \\
x\times x^{t_{ij}} &=&\left( x_{1}^{2}+x_{2}^{2}+x_{3}^{2}+x_{4}^{2}\right)
+2\left( x_{1}x_{4}-x_{2}x_{3}\right) ij\text{.}
\end{eqnarray*}

\section{Generalized Bicomplex Numbers}

In this section we define generalized bicomplex number and give some
algebraic properties of them.

\begin{definition}
A generalized bicomplex number $x$ is defined by the basis $\left \{
1,i,j,ij\right \} $ as follows%
\begin{equation*}
x=x_{1}1+x_{2}i+x_{3}j+x_{4}ij,
\end{equation*}%
where $x_{1},$ $x_{2},$ $x_{3}$ and $x_{4}$ are real numbers and $%
i^{2}=-\alpha ,$ $j^{2}=-\beta ,$ $\left( ij\right) ^{2}=\alpha \beta ,$ $%
ij=ji$, $\alpha ,\beta \in \mathbb{R}.$
\end{definition}

\begin{definition}
We denote the set of generalized bicomplex numbers by $C_{\alpha \beta }.$
For any $x=x_{1}1+x_{2}i+x_{3}j+x_{4}ij$ and $y=y_{1}1+y_{2}i+y_{3}j+y_{4}ij$
in $C_{\alpha \beta },$ the generalized bicomplex number addition is defined
as
\begin{equation*}
x+y=\left( x_{1}+y_{1}\right) +\left( x_{2}+y_{2}\right) i+\left(
x_{3}+y_{3}\right) j+\left( x_{4}+y_{4}\right) ij
\end{equation*}%
and the multiplication of a generalized bicomplex number $%
x=x_{1}1+x_{2}i+x_{3}j+x_{4}ij$ by a real scalar $\lambda $ is defined as%
\begin{equation*}
\lambda x=\lambda x_{1}1+\lambda x_{2}i+\lambda x_{3}j+\lambda x_{4}ij
\end{equation*}
\end{definition}

\begin{corollary}
The set of generalized bicomplex numbers $C_{\alpha \beta }$ is a real
vector space with this addition and scalar multiplication operations.
\end{corollary}

\begin{definition}
Generalized bicomplex number product, denoted by $\cdot $, over the set of
generalized bicomplex numbers $C_{\alpha \beta }$\ is given by
\begin{eqnarray*}
x\cdot y &=&\left( x_{1}y_{1}-\alpha x_{2}y_{2}-\beta x_{3}y_{3}+\alpha
\beta x_{4}y_{4}\right) +\left( x_{1}y_{2}+x_{2}y_{1}-\beta x_{3}y_{4}-\beta
x_{4}y_{3}\right) i \\
&&+\left( x_{1}y_{3}+x_{3}y_{1}-\alpha x_{2}y_{4}-\alpha x_{4}y_{2}\right)
j+\left( x_{1}y_{4}+x_{4}y_{1}+x_{2}y_{3}+x_{3}y_{2}\right) ij\text{.}
\end{eqnarray*}
\end{definition}

\begin{theorem}
Vector space $C_{\alpha \beta }$ together with the generalized bicomplex
product $\cdot $\ is a real algebra.
\end{theorem}

\begin{proof}
$\cdot :C_{\alpha \beta }\times C_{\alpha \beta }\rightarrow C_{\alpha \beta
}$ $\forall p,q,r\in C_{\alpha \beta }$ and $\lambda \in \mathbb{R}$ satisfy
the following conditions

i) $p\cdot (q+r)=p\cdot q+p\cdot r$

ii) $p\cdot (q\cdot r)=(p\cdot q)\times r$

iii) $(\lambda p)\cdot q=p\cdot \left( \lambda q\right) =\lambda \left(
p\cdot q\right) $

So, the real vector space $C_{\alpha \beta }$\ is a real algebra with
generalized bicomplex number product.
\end{proof}

Since the generalized bicomplex algebra is associative, it can be considered
in terms of matrices. Consider the set of matrices%
\begin{equation*}
Q_{\alpha \beta }=\left \{ \left(
\begin{array}{cccc}
x_{1} & -\alpha x_{2} & -\beta x_{3} & \alpha \beta x_{4} \\
x_{2} & x_{1} & -\beta x_{4} & -\beta x_{3} \\
x_{3} & -\alpha x_{4} & x_{1} & -\alpha x_{2} \\
x_{4} & x_{3} & x_{2} & x_{1}%
\end{array}%
\right) ;\text{ \  \  \  \  \  \ }x_{i}\in \mathbb{R}\text{ ,\  \  \  \ }1\leq i\leq
4\right \} \text{.}
\end{equation*}%
The set $Q_{\alpha \beta }$ together with matrix addition and scalar matrix
multiplication is a real vector space. Furthermore, the vector space
together with matrix product is an algebra.

\begin{theorem}
The algebras $C_{\alpha \beta }$ and $Q_{\alpha \beta }$ are isomorphic.
\end{theorem}

\begin{proof}
The transformation
\begin{equation*}
h:C_{\alpha \beta }\rightarrow Q_{\alpha \beta }
\end{equation*}%
given by
\begin{equation*}
h\left( x=x_{1}1+x_{2}i+x_{3}j+x_{4}ij\right) =\left(
\begin{array}{cccc}
x_{1} & -\alpha x_{2} & -\beta x_{3} & \alpha \beta x_{4} \\
x_{2} & x_{1} & -\beta x_{4} & -\beta x_{3} \\
x_{3} & -\alpha x_{4} & x_{1} & -\alpha x_{2} \\
x_{4} & x_{3} & x_{2} & x_{1}%
\end{array}%
\right)
\end{equation*}%
is one to one and onto. Morever $\forall x,y\in C_{\alpha \beta }$ and $%
\lambda \in \mathbb{R},$ we have
\begin{eqnarray*}
h\left( x+y\right) &=&h\left( x\right) +h\left( y\right) \\
h\left( \lambda x\right) &=&\lambda h\left( x\right) \\
h\left( xy\right) &=&h\left( x\right) h\left( y\right) \text{.}
\end{eqnarray*}%
Thus the algebras $C_{\alpha \beta }$and $Q_{\alpha \beta }$ are isomorphic.
\end{proof}

\begin{definition}
Let $x\in C_{\alpha \beta }.$ Then $x$ can be expressed as $x=\left(
x_{1}+x_{2}i\right) +\left( x_{3}+x_{4}i\right) j$. Conjugations of
generalized bicomplex numbers with respect to $i,j,ij$ are given by%
\begin{eqnarray*}
x^{t_{i}} &=&\left[ \left( x_{1}+x_{2}i\right) +\left( x_{3}+x_{4}i\right) j%
\right] ^{t_{i}}=\left( x_{1}-x_{2}i\right) +\left( x_{3}-x_{4}i\right) j \\
x^{t_{j}} &=&\left[ \left( x_{1}+x_{2}i\right) +\left( x_{3}+x_{4}i\right) j%
\right] ^{t_{j}}=\left( x_{1}+x_{2}i\right) -\left( x_{3}+x_{4}i\right) j \\
x^{t_{ij}} &=&\left[ \left( x_{1}+x_{2}i\right) +\left( x_{3}+x_{4}i\right) j%
\right] ^{t_{ij}}=\left( x_{1}-x_{2}i\right) -\left( x_{3}-x_{4}i\right) j%
\text{.}
\end{eqnarray*}%
where $x^{t_{i}},x^{t_{j}}$ and $x^{t_{ij}}$ denote conjugations of $x$ with
respect to $i,j,ij,$ respectively. Also we can compute
\begin{eqnarray*}
x\cdot x^{t_{i}} &=&\left( x_{1}^{2}+\alpha x_{2}^{2}-\beta x_{3}^{2}-\alpha
\beta x_{4}^{2}\right) +2\left( x_{1}x_{3}+\alpha x_{2}x_{4}\right) j \\
x\cdot x^{t_{j}} &=&\left( x_{1}^{2}-\alpha x_{2}^{2}+\beta x_{3}^{2}-\alpha
\beta x_{4}^{2}\right) +2\left( x_{1}x_{2}+\beta x_{3}x_{4}\right) i \\
x\cdot x^{t_{ij}} &=&\left( x_{1}^{2}+\alpha x_{2}^{2}+\beta
x_{3}^{2}+\alpha \beta x_{4}^{2}\right) +2\left(
x_{1}x_{4}-x_{2}x_{3}\right) ij\text{.}
\end{eqnarray*}
\end{definition}

\begin{proposition}
Conjugations of generalized bicomplex numbers with respect to $i,j,ij$ have
following properties

i) $\left( \lambda p+\delta q\right) ^{t_{k}}=\lambda p^{t_{k}}+\delta
q^{t_{k}}$

ii) $\left( p^{t_{k}}\right) ^{t_{k}}=p$

iii) $\left( p\cdot q\right) ^{t_{k}}=p^{t_{k}}\cdot q^{t_{k}},$

where $p,q\in C_{\alpha \beta },$ $\lambda ,\delta \in \mathbb{R}$ and $%
t_{k} $ represent the conjugations with respect to $i,j$ and $ij.$
\end{proposition}

\begin{proof}
The proofs of the properties can be easily seen by directly computation.
\end{proof}

\section{Some Hyperquadrics and Lie Groups}

In this section we show that some hyperquadrics together with generalized
bicomplex number product are Lie groups and find their Lie algebras. We deal
with the hyperquadric $M_{t_{i}}$
\begin{equation*}
M_{t_{i}}=\left \{ x=\left( x_{1},x_{2},x_{3},x_{4}\right) \in \mathbb{R}%
_{v}^{4}:\text{ }x_{1}x_{3}+\alpha x_{2}x_{4}=0,\text{\ }g_{t_{i}}(x,x)\neq
0\right \}
\end{equation*}%
We consider $M_{t_{i}}$ as the set of generalized bicomplex numbers%
\begin{equation*}
M_{t_{i}}=\left \{ x=x_{1}1+x_{2}i+x_{3}j+x_{4}ij\in \mathbb{R}_{v}^{4}:\text{
}x_{1}x_{3}+\alpha x_{2}x_{4}=0,\text{\ }g_{t_{i}}(x,x)\neq 0\right \}
\end{equation*}%
The components of $M_{t_{i}}$ are easily obtained by representing
generalized bicomplex number multiplication in matrix form%
\begin{equation*}
\tilde{M}_{t_{i}}=\left \{ x=\left(
\begin{array}{cccc}
x_{1} & -\alpha x_{2} & -\beta x_{3} & \alpha \beta x_{4} \\
x_{2} & x_{1} & -\beta x_{4} & -\beta x_{3} \\
x_{3} & -\alpha x_{4} & x_{1} & -\alpha x_{2} \\
x_{4} & x_{3} & x_{2} & x_{1}%
\end{array}%
\right) ,\text{ }x_{1}x_{3}+\alpha x_{2}x_{4}=0,\text{ }g_{t_{i}}(x,x)\neq
0\right \}
\end{equation*}%
where $g_{t_{i}}$ is Euclidean or pseudo-Euclidean metric and it is defined
by $g_{t_{i}}=dx_{1}^{2}+\alpha dx_{2}^{2}-\beta dx_{3}^{2}-\alpha \beta
dx_{4}^{2}.$

\begin{remark}
The norm of any element $x$ on the hyperquadric $M_{t_{i}}$ is given by $%
N_{x}=x\cdot x^{t_{i}}=g_{t_{i}}(x,x).$
\end{remark}

Now we define the hyperquadric $M_{t_{j}}$ as%
\begin{equation*}
M_{t_{j}}=\left \{ x=\left( x_{1},x_{2},x_{3},x_{4}\right) \in \mathbb{R}%
_{v}^{4}:\text{ }x_{1}x_{2}+\beta x_{3}x_{4}=0,\text{\ }g_{t_{j}}(x,x)\neq
0\right \}
\end{equation*}%
We consider $M_{t_{j}}$ as the set of generalized bicomplex numbers%
\begin{equation*}
M_{t_{j}}=\left \{ x=x_{1}1+x_{2}i+x_{3}j+x_{4}ij\in \mathbb{R}_{v}^{4}:\text{
}x_{1}x_{2}+\beta x_{3}x_{4}=0,\text{\ }g_{t_{j}}(x,x)\neq 0\right \}
\end{equation*}%
The components of $M_{t_{j}}$ are easily obtained by representing
generalized bicomplex number multiplication in matrix form%
\begin{equation*}
\tilde{M}_{t_{j}}=\left \{ x=\left(
\begin{array}{cccc}
x_{1} & -\alpha x_{2} & -\beta x_{3} & \alpha \beta x_{4} \\
x_{2} & x_{1} & -\beta x_{4} & -\beta x_{3} \\
x_{3} & -\alpha x_{4} & x_{1} & -\alpha x_{2} \\
x_{4} & x_{3} & x_{2} & x_{1}%
\end{array}%
\right) ,\text{ }x_{1}x_{2}+\beta x_{3}x_{4}=0,\text{ }g_{t_{j}}(x,x)\neq
0\right \}
\end{equation*}%
where $g_{t_{j}}$ is Euclidean or pseudo-Euclidean metric and it is defined
by $g_{t_{j}}=dx_{1}^{2}-\alpha dx_{2}^{2}+\beta dx_{3}^{2}-\alpha \beta
dx_{4}^{2}.$

\begin{remark}
The norm of any element $x$ on the hyperquadric $M_{t_{j}}$ is given by $%
N_{x}=x\cdot x^{t_{j}}=g_{t_{j}}(x,x).$
\end{remark}

We define the hyperquadric $M_{t_{ij}}$%
\begin{equation*}
M_{t_{ij}}=\left \{ x=\left( x_{1},x_{2},x_{3},x_{4}\right) \in \mathbb{R}%
_{v}^{4}:\text{ }x_{1}x_{4}-x_{2}x_{3}=0,\text{\ }g_{t_{ij}}(x,x)\neq
0\right \}
\end{equation*}%
We consider $M_{t_{ij}}$ as the set of generalized bicomplex numbers%
\begin{equation*}
M_{t_{ij}}=\left \{ x=x_{1}1+x_{2}i+x_{3}j+x_{4}ij\in \mathbb{R}_{v}^{4}:%
\text{ }x_{1}x_{4}-x_{2}x_{3}=0,\text{\ }g_{t_{ij}}(x,x)\neq 0\right \}
\end{equation*}%
The components of $M_{t_{ij}}$ are easily obtained by representing
generalized bicomplex number multiplication in matrix form%
\begin{equation*}
\tilde{M}_{t_{ij}}=\left \{ x=\left(
\begin{array}{cccc}
x_{1} & -\alpha x_{2} & -\beta x_{3} & \alpha \beta x_{4} \\
x_{2} & x_{1} & -\beta x_{4} & -\beta x_{3} \\
x_{3} & -\alpha x_{4} & x_{1} & -\alpha x_{2} \\
x_{4} & x_{3} & x_{2} & x_{1}%
\end{array}%
\right) ,\text{ }x_{1}x_{4}-x_{2}x_{3}=0,\text{ }g_{t_{ij}}(x,x)\neq
0\right \}
\end{equation*}%
where $g_{t_{ij}}$ is Euclidean or pseudo-Euclidean metric and it is defined
by $g_{t_{ij}}=dx_{1}^{2}+\alpha dx_{2}^{2}+\beta dx_{3}^{2}+\alpha \beta
dx_{4}^{2}.$

\begin{remark}
The norm of any element $x$ on the hyperquadric $M_{t_{ij}}$ is given by $%
N_{x}=x\cdot x^{t_{ij}}=g_{t_{ij}}(x,x).$
\end{remark}

\begin{theorem}
The set of $M_{t_{i}}$ with generalized bicomplex number product is a Lie
group.
\end{theorem}

\begin{proof}
$\tilde{M}_{t_{i}}$ differentiable manifold and at the same time $\tilde{M}%
_{t_{i}}$ is a group with group operation given by matrix multiplication.
The group function is given by
\begin{equation*}
.:\tilde{M}_{t_{i}}\times \tilde{M}_{t_{i}}\rightarrow \tilde{M}_{t_{i}}
\end{equation*}%
\begin{equation*}
\left( x,y\right) \rightarrow x.y^{-1},
\end{equation*}%
where $y^{-1}$ is obtained as a element of $M_{t_{i}}$ as follows:
\begin{equation*}
y^{-1}=\frac{y^{t_{i}}}{N_{y}}=\frac{1}{y_{1}^{2}+\alpha y_{2}^{2}-\beta
y_{3}^{2}-\alpha \beta y_{4}^{2}.}\left( y_{1},-y_{2},y_{3},-y_{4}\right) .
\end{equation*}%
Since the transformation $h$ is an isomorphism $\left( M_{t_{i}},\cdot
\right) $ is a Lie group.
\end{proof}

We denote the set of all unit generalized bicomplex numbers $x$ on $M_{t_{i}}
$ by $M_{t_{i}}^{\ast }.$ $M_{t_{i}}^{\ast }$ is defined as$\ $
\begin{equation*}
M_{t_{i}}^{\ast }=\left \{ x\in M_{t_{i}}:g_{t_{i}}\left( x,x\right)
=1\right \}
\end{equation*}%
or%
\begin{equation*}
M_{t_{i}}^{\ast }=\left \{ x\in M_{t_{i}}:x_{1}^{2}+\alpha x_{2}^{2}-\beta
x_{3}^{2}-\alpha \beta x_{4}^{2}=1\right \}
\end{equation*}%
$M_{t_{i}}^{\ast }$ is a group with the group operation of generalized
bicomplex multiplication. So we can give the following corollary.

\begin{corollary}
$M_{t_{i}}^{\ast }$ is 2-dimensional Lie subgroup of $M_{t_{i}}.$
\end{corollary}

\begin{theorem}
The Lie algebra of Lie group $M_{t_{i}}$ is $sp\left \{
X_{1},X_{2,}X_{4}\right \} $ such that left invariant vector fields $%
X_{1},X_{2,}X_{4}$ are given by%
\begin{eqnarray*}
X_{1} &=&\left( x_{1},x_{2},x_{3},x_{4}\right)  \\
X_{2} &=&\left( -\alpha x_{2},x_{1},-\alpha x_{4},x_{3}\right)  \\
X_{4} &=&\left( \alpha \beta x_{4},-\beta x_{3},-\alpha x_{2},x_{1}\right)
\end{eqnarray*}
\end{theorem}

\begin{proof}
Let us find the Lie algebra of Lie group $M_{t_{i}}.$ Let
\begin{equation*}
a\left( t\right) =a_{1}\left( t\right) 1+a_{2}\left( t\right) i+a_{3}\left(
t\right) j+a_{4}\left( t\right) ij
\end{equation*}%
be a curve on $M_{t_{i}}$ such that $a\left( 0\right) =1,$ i.e. $a_{1}\left(
0\right) =1,$ $a_{m}\left( 0\right) =0$ for $m=2,3,4.$ Differentiation of
the equation%
\begin{equation*}
a_{1}\left( t\right) a_{3}\left( t\right) +\alpha a_{2}\left( t\right)
a_{4}\left( t\right) =0
\end{equation*}%
yields the equation%
\begin{equation*}
a_{1}^{\prime }\left( t\right) a_{3}\left( t\right) +a_{1}\left( t\right)
a_{3}^{\prime }\left( t\right) +\alpha a_{2}^{\prime }\left( t\right)
a_{4}\left( t\right) +\alpha a_{2}\left( t\right) a_{4}^{\prime }\left(
t\right) =0
\end{equation*}%
Substituting $t=0$, we obtain $a_{3}^{\prime }\left( 0\right) =0.$ The Lie
algebra is thus constituted by vectors of the form $\zeta =\left. \zeta
_{m}\left( \frac{\partial }{\partial a_{m}}\right) \right \vert _{\alpha =1}$
where $m=1,2,4.$ The vector $\zeta $ is formally written in the form $\zeta
=\zeta _{1}+\zeta _{2}j+\zeta _{4}ij.$ Let us find the left invariant vector
field $X$ on $M_{t_{i}}$ for which $\left. X\right \vert _{\alpha =1}=\zeta
. $ Let $b\left( t\right) $ be a curve on $M_{t_{i}}$ such that $b\left(
0\right) =1,$ $b^{\prime }\left( 0\right) =\zeta .$ Then $L_{x}\left(
b\left( t\right) \right) =xb\left( t\right) $ is the left translation of the
curve $b\left( t\right) $ by the generalized bicomplex number $x.$ Let $%
L_{x}^{\ast }$ be the differentiation of $L_{x}$ left translation. In that
case $L_{x}^{\ast }\left( b^{\prime }\left( 0\right) \right) =x\zeta .$ In
particular, denote by $X_{m}$ those left invariant vector fields on $%
M_{t_{i}}$ for which%
\begin{equation*}
\left. X_{m}\right \vert _{\alpha =1}=\left. \frac{\partial }{\partial a_{m}}%
\right \vert _{\alpha =1}
\end{equation*}%
where $m=1,2,4.$ These three vector fields are represented at the point $%
\alpha =1$ by the generalized bicomplex units $1,i,ij$. For the components
of these vector fields at the point $x=x_{1}1+x_{2}i+x_{3}j+x_{4}ij,$ we
have $\left( X_{1}\right) _{x}=x1,$ $\left( X_{2}\right) _{x}=xi$ and $%
\left( X_{4}\right) _{x}=xij.$%
\begin{eqnarray*}
X_{1} &=&\left( x_{1},x_{2},x_{3},x_{4}\right) , \\
X_{2} &=&\left( -\alpha x_{2},x_{1},-\alpha x_{4},x_{3}\right) , \\
X_{4} &=&\left( \alpha \beta x_{4},-\beta x_{3},-\alpha x_{2},x_{1}\right) ,
\end{eqnarray*}%
where all the partial derivaties are at the point $x.$
\end{proof}

\begin{corollary}
The Lie algebra of Lie group $M_{t_{i}}^{\ast }$ is $sp\left \{
X_{2,}X_{4}\right \} $.
\end{corollary}

\begin{theorem}
The set of $M_{t_{j}}$ together with generalized bicomplex number product is
a Lie group.
\end{theorem}

\begin{corollary}
$M_{t_{j}}^{\ast }$ is 2-dimensional Lie subgroup of $M_{t_{j}}.$
\end{corollary}

\begin{theorem}
The Lie algebra of Lie group $M_{t_{j}}$ is $sp\left \{
X_{1},X_{3,}X_{4}\right \} $ such that left invariant vector fields $%
X_{1},X_{3,}X_{4}\ $are given by
\begin{eqnarray*}
X_{1} &=&\left( x_{1},x_{2},x_{3},x_{4}\right) , \\
X_{3} &=&\left( -\beta x_{3},-\beta x_{4},x_{1},x_{2}\right) , \\
X_{4} &=&\left( \alpha \beta x_{4},-\beta x_{3},-\alpha x_{2},x_{1}\right) .
\end{eqnarray*}
\end{theorem}

\begin{corollary}
The Lie algebra of Lie group $M_{t_{j}}^{\ast }$ is $sp\left \{
X_{3,}X_{4}\right \} $
\end{corollary}

\begin{theorem}
The set of $M_{t_{ij}}$ together with generalized bicomplex number product
is a Lie group.
\end{theorem}

\begin{corollary}
$M_{t_{ij}}^{\ast }$ is 2-dimensional Lie subgroup of $M_{t_{ij}}.$
\end{corollary}

\begin{theorem}
The Lie algebra of Lie group $M_{t_{ij}}$ is $sp\left \{
X_{1},X_{2,}X_{3}\right \} $ such that left invariant vector fields $%
X_{1},X_{2,}X_{3}$ are given by$\ $%
\begin{eqnarray*}
X_{1} &=&\left( x_{1},x_{2},x_{3},x_{4}\right) , \\
X_{2} &=&\left( -\alpha x_{2},x_{1},-\alpha x_{4},x_{3}\right) , \\
X_{3} &=&\left( -\beta x_{3},-\beta x_{4},x_{1},x_{2}\right) .
\end{eqnarray*}
\end{theorem}

\begin{corollary}
The Lie algebra of Lie group $M_{t_{ij}}^{\ast }$ is $sp\left \{
X_{2,}X_{3}\right \} .$
\end{corollary}

\section{Tensor Product Surfaces and Lie Groups}

In this section we define the tensor product surfaces on the hyperquadrics $%
M_{t_{i}},$\ $M_{t_{j}}$ and $M_{t_{ij}}.$\ By means of tensor product
surfaces, we determine some special subgroups of these Lie groups $%
M_{t_{i}}, $\ $M_{t_{j}}$ and $M_{t_{ij}}$ in $\mathbb{R}^{4}$ and $\mathbb{R%
}_{2}^{4}$.

\subsection{Tensor Product Surfaces on $M_{t_{i}}$ Hyperquadric and Some
Special Lie Subgroups}

In this subsection, we change the definition of tensor product as follows:

Let $\gamma :\mathbb{R\rightarrow R}_{k}^{2}$ $\left( +\text{ }-\alpha \beta
\right) $ and $\delta :\mathbb{R\rightarrow R}_{t}^{2}$ $\left( +\text{ }%
\alpha \right) $ be planar curves in Euclidean or Lorentzian space. Put $%
\gamma \left( t\right) =\left( \gamma _{1}\left( t\right) ,\gamma _{2}\left(
t\right) \right) $ and $\delta \left( s\right) =\left( \delta _{1}\left(
s\right) ,\delta _{2}\left( s\right) \right) .$ Let us define their tensor
product as%
\begin{equation*}
f=\gamma \otimes \delta :\mathbb{R}^{2}\rightarrow \mathbb{R}_{v}^{4}\text{
\  \ }\left( +\text{ }\alpha \text{ }-\beta \text{ }-\alpha \beta \right) ,
\end{equation*}%
\begin{equation}
f\left( t,s\right) =\left( \gamma _{1}\left( t\right) \delta _{1}\left(
s\right) ,\gamma _{1}\left( t\right) \delta _{2}\left( s\right) ,-\alpha
\gamma _{2}\left( t\right) \delta _{2}\left( s\right) ,\gamma _{2}\left(
t\right) \delta _{1}\left( s\right) \right) .
\end{equation}%
Tensor product surface given by (1) is a surface on $M_{t_{i}}$
hyperquadric. Tangent vector fields of $f\left( t,s\right) $ can be easily
computed as%
\begin{eqnarray}
\frac{\partial f}{\partial t} &=&\left( \gamma _{1}^{\prime }\left( t\right)
\delta _{1}\left( s\right) ,\gamma _{1}^{\prime }\left( t\right) \delta
_{2}\left( s\right) ,-\alpha \gamma _{2}^{\prime }\left( t\right) \delta
_{2}\left( s\right) ,\gamma _{2}^{\prime }\left( t\right) \delta _{1}\left(
s\right) \right)  \\
\frac{\partial f}{\partial s} &=&\left( \gamma _{1}\left( t\right) \delta
_{1}^{\prime }\left( s\right) ,\gamma _{1}\left( t\right) \delta
_{2}^{\prime }\left( s\right) ,-\alpha \gamma _{2}\left( t\right) \delta
_{2}^{\prime }\left( s\right) ,\gamma _{2}\left( t\right) \delta
_{1}^{\prime }\left( s\right) \right) .  \notag
\end{eqnarray}%
By using (2), we have
\begin{eqnarray*}
g_{11} &=&g\left( \frac{\partial f}{\partial t},\frac{\partial f}{\partial t}%
\right) =g_{1}(\gamma ^{\prime },\gamma ^{\prime })g_{2}\left( \delta
,\delta \right)  \\
g_{12} &=&g\left( \frac{\partial f}{\partial t},\frac{\partial f}{\partial s}%
\right) =g_{1}(\gamma ,\gamma ^{\prime })g_{2}\left( \delta ,\delta ^{\prime
}\right)  \\
g_{22} &=&g\left( \frac{\partial f}{\partial s},\frac{\partial f}{\partial s}%
\right) =g_{1}(\gamma ,\gamma )g_{2}\left( \delta ^{\prime },\delta ^{\prime
}\right) ,
\end{eqnarray*}%
where $g_{1}=dx_{1}^{2}-\alpha \beta dx_{2}^{2}$ and $g_{2}=dx_{1}^{2}+%
\alpha dx_{2}^{2}$ are the metrics of $\mathbb{R}_{k}^{2}$ and $\mathbb{R}%
_{t}^{2}$, respectively. Consequently, an orthonormal basis for the tangent
space of $f(t,s)$ is given by%
\begin{eqnarray*}
e_{1} &=&\frac{1}{\sqrt{\left \vert g_{11}\right \vert }}\frac{\partial f}{%
\partial t} \\
e_{2} &=&\frac{1}{\sqrt{\left \vert g_{11}\left(
g_{11}g_{22}-g_{12}^{2}\right) \right \vert }}\left( g_{11}\frac{\partial f}{%
\partial s}-g_{12}\frac{\partial f}{\partial t}\right)
\end{eqnarray*}

\begin{remark}
Tensor product surface given by (1) is a surface in $\mathbb{R}^{4}$ or $%
\mathbb{R}_{2}^{4}$ according to the case of $\alpha $ and $\beta .$ If we
take as $\alpha =\beta =1,$ we obtain a tensor product surface of a
Lorentzian plane curve and a Euclidean plane curve in $\mathbb{R}_{2}^{4}.$
If we take as $\alpha =1$, $\beta =-1,$ we obtain a tensor product surface
of two Euclidean plane curves in $\mathbb{R}^{4}.$ If we take as $\alpha =-1$%
, $\beta =1,$ we obtain a tensor product surface of a Euclidean plane curve
and a Lorentzian plane curve in $\mathbb{R}_{2}^{4}.$ If we take as $\alpha
=-1$, $\beta =-1$ we obtain a tensor product surface of two Lorentzian plane
curves in $\mathbb{R}_{2}^{4}.$
\end{remark}

Now we investigate Lie group structure of tensor product surfaces given by
the parametrization (1) in $\mathbb{R}^{4}$ or $\mathbb{R}_{2}^{4}$
according to above cases. Morever we obtain left invariant vector fields of
the tensor product surface that has the structure of Lie group.

\subsubsection{Case I: $\protect \alpha =\protect \beta =1$}

\begin{proposition}
Let $\gamma :\mathbb{R\rightarrow R}_{1}^{2}$ $\left( +\text{ }-\right) $ be
a hyperbolic spiral and $\delta :\mathbb{R\rightarrow R}^{2}$ $\left( +\text{
}+\right) $ be a spiral with the same parameter, i.e. $\gamma \left(
t\right) =e^{at}\left( \cosh t,\sinh t\right) $ and $\delta \left( t\right)
=e^{bt}\left( \cos t,\sin t\right) $ $\left( a,b\in \mathbb{R}\right) .$
Their tensor product is a one parameter subgroup of Lie group $M_{t_{i}}.$
\end{proposition}

\begin{proof}
We obtain
\begin{equation*}
\varphi \left( t\right) =\gamma \left( t\right) \otimes \delta \left(
t\right) =e^{(a+b)t}\left( \cosh t\cos t,\cosh t\sin t,-\sinh t\sin t,\sinh
t\cos t\right)
\end{equation*}%
It can be easily seen that
\begin{equation*}
\varphi \left( t_{1}\right) \cdot \varphi \left( t_{2}\right) =\varphi
\left( t_{1}+t_{2}\right)
\end{equation*}%
for all $t_{1},t_{2}.$ Also $\varphi ^{-1}\left( t\right) =\varphi \left(
-t\right) .$ Hence $\left( \varphi \left( t\right) ,\cdot \right) $ is a one
parameter Lie subgroup of Lie group $\left( M_{t_{i}},\cdot \right) .$
\end{proof}

\begin{corollary}
Let $\gamma :\mathbb{R\rightarrow R}_{1}^{2}$ be a Lorentzian circle
centered at O and $\delta :\mathbb{R\rightarrow R}^{2}$ be circle centered
at O with the same parameter, i.e., $\gamma \left( t\right) =\left( \cosh
t,\sinh t\right) $ and $\delta \left( t\right) =\left( \cos t,\sin t\right)
. $ Then their tensor product is a one parameter subgroup of Lie group $%
M_{t_{i}}^{\ast }.$

\begin{proof}
Since $g_{t_{i}}\left( \gamma \left( t\right) \otimes \delta \left( t\right)
,\gamma \left( t\right) \otimes \delta \left( t\right) \right) =1,$ it
follows that $\gamma \left( t\right) \otimes \delta \left( t\right) \subset
M_{t_{i}}^{\ast }.$ If we take as $a=b=0$ in Theorem (), we obtain that $%
\gamma $ is a Lorentzian cirle centered at O and $\delta $ is a circle
centered at O. Then their tensor product is a one parameter Lie subgroup in
Lie group $M_{t_{i}}^{\ast }.$
\end{proof}
\end{corollary}

\begin{proposition}
Let $\varphi \left( t\right) $ be tensor product of a Lorentzian cirle
centered  at O and a circle centered at O with the same parameter. Then the
left invariant vector field on $\varphi \left( t\right) $ is $X=X_{2}+X_{4},$
where $X_{2}$ and $X_{4}$ are left invariant vector fields on $%
M_{t_{i}}^{\ast }.$
\end{proposition}

\begin{proof}
Since $\varphi \left( t\right) $ is tensor product of a Lorentzian cirle
centered at O and a circle centered at O with the same parameter we write%
\begin{equation*}
\varphi \left( t\right) =\left( \cosh t\cos t,\cosh t\sin t,-\sinh t\sin
t,\sinh t\cos t\right)
\end{equation*}%
$\varphi \left( 0\right) =(1,0,0,0)=e$ and $\varphi ^{\prime }\left(
0\right) =\left( 0,1,0,1\right) =X_{e}.$ Then
\begin{eqnarray*}
L_{g}^{\ast }\left( X_{e}\right)  &=&g\cdot X_{e}=\left(
x_{1}1+x_{2}i+x_{3}j+x_{4}ij\right) \cdot \left( i+ij\right)  \\
&=&X_{2}+X_{4}
\end{eqnarray*}%
This completes the proof.
\end{proof}

\begin{proposition}
Let $\gamma :\mathbb{R\rightarrow R}_{1}^{2}$, $\gamma \left( t\right)
=e^{at}\left( \cosh t,\sinh t\right) $ be a hyperbolic spiral and $\delta :%
\mathbb{R\rightarrow R}^{2}$ $\delta \left( s\right) =e^{bs}\left( \cos
s,\sin s\right) $ be a spiral $\left( a,b\in \mathbb{R}\right) .$ Then their
tensor product is 2-dimensional Lie subgroup of $M_{t_{i}}.$

\begin{proof}
By using tensor product rule given by (1), we get%
\begin{equation*}
f\left( t,s\right) =e^{at+bs}\left( \cosh t\cos s,\cosh t\sin s,-\sinh t\sin
s,\sinh t\cos s\right)
\end{equation*}%
Every point of $f\left( t,s\right) $\ is on the hyperquadric $M_{t_{i}}.$
Since $f\left( t,s\right) $ is both subgroup and submanifold of a Lie group $%
M_{t_{i}},$ it is a 2-dimensional Lie subgroup of $M_{t_{i}}.$
\end{proof}
\end{proposition}

\begin{proposition}
Let $\gamma :\mathbb{R\rightarrow R}_{1}^{2}$, $\gamma \left( t\right)
=\left( \cosh t,\sinh t\right) $ be a Lorentzian circle centered at O and $%
\delta :\mathbb{R\rightarrow R}^{2}$ $\delta \left( s\right) =\left( \cos
s,\sin s\right) $ be a circle centered at O $\left( a,b\in \mathbb{R}\right)
.$ Then their tensor product is 2-dimensional Lie subgroup of $%
M_{t_{i}}^{\ast }.$

\begin{proof}
In Proposition () taking $a=b=0$, we can see that the tensor product surface
$f\left( t,s\right) \subset M_{t_{i}}^{\ast }.$Hence, it is 2-dimensional
Lie subgroup of $M_{t_{i}}^{\ast }$.
\end{proof}
\end{proposition}

\begin{proposition}
Let $\gamma :\mathbb{R\rightarrow R}_{1}^{2}$, $\gamma \left( t\right)
=\left( \cosh t,\sinh t\right) $ be a Lorentzian circle at centered O and $%
\delta :\mathbb{R\rightarrow R}^{2}$ $\delta \left( s\right) =\left( \cos
s,\sin s\right) $ be a circle at centered O $\left( a,b\in \mathbb{R}\right)
.$ Then, the left invariant vector fields on tensor product surface $f\left(
t,s\right) =\gamma \left( t\right) \otimes \delta \left( s\right) $ are $%
X_{2}$ and $X_{4}$ which are the left invariant vector fields on $%
M_{t_{i}}^{\ast }.$
\end{proposition}

\begin{proof}
The unit element of 2-dimensional Lie subgroup is the point $e=f(0,0).$ Let
us find the left invariant vector fields on $f(t,s)$ to the vectors%
\begin{equation*}
u_{1}=\left. \frac{\partial }{\partial t}\right \vert _{e}\text{ \  \ and \  \
}u_{2}=\left. \frac{\partial }{\partial s}\right \vert _{e}
\end{equation*}%
for the vector $u_{1}$ we obtain

\begin{eqnarray*}
L_{g}^{\ast }\left( u_{1}\right) &=&g\cdot u_{1}=\left(
x_{1}1+x_{2}i+x_{3}j+x_{4}ij\right) \cdot ij \\
&=&X_{4}
\end{eqnarray*}%
Analogously, for the vector $u_{2}$ we obtain left invariant vector field $%
X_{2}.$
\end{proof}

\subsubsection{Case II\textbf{:} $\protect \alpha =1,\protect \beta =-1$}

\begin{proposition}
Let $\gamma :\mathbb{R\rightarrow R}^{2}$ $\left( +\text{ }+\right) $ and $%
\delta :\mathbb{R\rightarrow R}^{2}$ $\left( +\text{ }+\right) $ be two
spirals with the same parameter, i.e. $\gamma \left( t\right) =e^{at}\left(
\cos t,\sin t\right) $ and $\delta \left( t\right) =e^{bt}\left( \cos t,\sin
t\right) $ $\left( a,b\in \mathbb{R}\right) .$ Then their tensor product is
a one parameter subgroup of Lie group $M_{t_{i}}.$
\end{proposition}

\begin{proof}
We obtain
\begin{equation*}
\varphi \left( t\right) =\gamma \left( t\right) \otimes \delta \left(
t\right) =e^{(a+b)t}\left( \cos ^{2}t,\cos t,\sin t,-\sin ^{2}t,\cos t,\sin
t\right)
\end{equation*}%
It can be easily seen that
\begin{equation*}
\varphi \left( t_{1}\right) \cdot \varphi \left( t_{2}\right) =\varphi
\left( t_{1}+t_{2}\right)
\end{equation*}%
for all $t_{1},t_{2}.$ Also $\varphi ^{-1}\left( t\right) =\varphi \left(
-t\right) .$ Hence $\left( \varphi \left( t\right) ,\cdot \right) $ is a one
parameter Lie subgroup of Lie group $\left( M_{t_{i}},\cdot \right) .$
\end{proof}

\begin{corollary}
Let $\gamma :\mathbb{R\rightarrow R}^{2}$ and $\delta :\mathbb{R\rightarrow R%
}^{2}$ be two circles centered at O with the same parameter, i.e., $\gamma
\left( t\right) =\left( \cos t,\sin t\right) $ and $\delta \left( t\right)
=\left( \cos t,\sin t\right) .$ Then their tensor product is a one parameter
subgroup of Lie group $M_{t_{i}}^{\ast }.$
\end{corollary}

\begin{proposition}
Let $\varphi \left( t\right) $ be tensor product of two circles centered at
O with the same parameter. Then the left invariant vector field on $\varphi
\left( t\right) $ is $X=X_{2}+X_{4},$ where $X_{2}$ and $X_{4}$ are left
invariant vector fields on $M_{t_{i}}^{\ast }.$
\end{proposition}

\begin{proposition}
Let $\gamma :\mathbb{R\rightarrow R}^{2}$, $\gamma \left( t\right)
=e^{at}\left( \cos t,\sin t\right) $ and $\delta :\mathbb{R\rightarrow R}%
^{2} $ $\delta \left( s\right) =e^{bs}\left( \cos s,\sin s\right) $ be two
spirals $\left( a,b\in \mathbb{R}\right) .$ Then their tensor product is
2-dimensional Lie subgroup of $M_{t_{i}}.$

\begin{proof}
By using tensor product rule given by (1), we get%
\begin{equation*}
f\left( t,s\right) =e^{at+bs}\left( \cos t\cos s,\cos t\sin s,-\sin t\sin
s,\sin t\cos s\right)
\end{equation*}%
Every point of $f\left( t,s\right) $\ is on the hyperquadric $M_{t_{i}}.$
Since $f\left( t,s\right) $ is both subgroup and submanifold of a Lie group $%
M_{t_{i}},$ it is a 2-dimensional Lie subgroup of $M_{t_{i}}.$
\end{proof}
\end{proposition}

\begin{corollary}
Let $\gamma :\mathbb{R\rightarrow R}^{2}$, $\gamma \left( t\right) =\left(
\cos t,\sin t\right) $ and $\delta :\mathbb{R\rightarrow R}^{2}$ $\delta
\left( s\right) =\left( \cos s,\sin s\right) $ be two circles centered at O $%
\left( a,b\in \mathbb{R}\right) .$ Then their tensor product is
2-dimensional Lie subgroup of $M_{t_{i}}^{\ast }.$
\end{corollary}

\begin{proposition}
Let $\gamma :\mathbb{R\rightarrow R}_{1}^{2}$, $\gamma \left( t\right)
=\left( \cos t,\sin t\right) $ and $\delta :\mathbb{R\rightarrow R}^{2}$ $%
\delta \left( s\right) =\left( \cos s,\sin s\right) $ be two circles
centered at O $\left( a,b\in \mathbb{R}\right) .$ Then, the left invariant
vector fields on tensor product surface $f\left( t,s\right) =\gamma \left(
t\right) \otimes \delta \left( s\right) $ are $X_{2}$ and $X_{4}$ which are
the left invariant vector fields on $M_{t_{i}}^{\ast }.$
\end{proposition}

\subsubsection{Case III: $\protect \alpha =-1,\protect \beta =-1$}

\begin{proposition}
Let $\gamma :\mathbb{R\rightarrow R}_{1}^{2}$ $\left( +\text{ }-\right) $
and $\delta :\mathbb{R\rightarrow R}_{1}^{2}$ $\left( +\text{ }-\right) $ be
two hyperbolic spirals with the same parameter, i.e. $\gamma \left( t\right)
=e^{at}\left( \cosh t,\sinh t\right) $ and $\delta \left( t\right)
=e^{bt}\left( \cosh t,\sinh t\right) $ $\left( a,b\in \mathbb{R}\right) .$
Then their tensor product is a one parameter subgroup of Lie group $%
M_{t_{i}}.$
\end{proposition}

\begin{proof}
We obtain
\begin{equation*}
\varphi \left( t\right) =\gamma \left( t\right) \otimes \delta \left(
t\right) =e^{(a+b)t}\left( \cosh ^{2}t,\cosh t\sinh t,\sinh ^{2}t,\cosh
t\sinh t\right)
\end{equation*}%
It can be easily seen that
\begin{equation*}
\varphi \left( t_{1}\right) \cdot \varphi \left( t_{2}\right) =\varphi
\left( t_{1}+t_{2}\right)
\end{equation*}%
for all $t_{1},t_{2}.$ Also $\varphi ^{-1}\left( t\right) =\varphi \left(
-t\right) .$ Hence $\left( \varphi \left( t\right) ,\cdot \right) $ is a one
parameter Lie subgroup of Lie group $\left( M_{t_{i}},\cdot \right) .$
\end{proof}

\begin{corollary}
Let $\gamma :\mathbb{R\rightarrow R}_{1}^{2}$ and $\delta :\mathbb{%
R\rightarrow R}_{1}^{2}$ be two Lorentzian circles centered at O with the
same parameter, i.e., $\gamma \left( t\right) =\left( \cosh t,\sinh t\right)
$ and $\delta \left( t\right) =\left( \cosh t,\sinh t\right) .$ Then their
tensor product is a one parameter subgroup of Lie group $M_{t_{i}}^{\ast }.$
\end{corollary}

\begin{proposition}
Let $\varphi \left( t\right) $ be tensor product of two Lorentzian circles
centered at O with the same parameter. Then the left invariant vector field
on $\varphi \left( t\right) $ is $X=X_{2}+X_{4},$ where $X_{2}$ and $X_{4}$
are left invariant vector fields on $M_{t_{i}}^{\ast }.$
\end{proposition}

\begin{proposition}
Let $\gamma :\mathbb{R\rightarrow R}_{1}^{2}$, $\gamma \left( t\right)
=e^{at}\left( \cosh t,\sinh t\right) $ and $\delta :\mathbb{R\rightarrow R}%
_{1}^{2}$ $\delta \left( s\right) =e^{bs}\left( \cosh s,\sinh s\right) $ be
two spirals $\left( a,b\in \mathbb{R}\right) .$ Then their tensor product is
2-dimensional Lie subgroup of $M_{t_{i}}.$

\begin{proof}
By using tensor product rule given by (1), we get%
\begin{equation*}
f\left( t,s\right) =e^{at+bs}\left( \cosh t\cosh s,\cosh t\sinh s,-\sinh
t\sinh s,\sinh t\cosh s\right)
\end{equation*}%
Every point of $f\left( t,s\right) $\ is on the hyperquadric $M_{t_{i}}.$
Since $f\left( t,s\right) $ is both subgroup and submanifold of a Lie group $%
M_{t_{i}},$ it is a 2-dimensional Lie subgroup of $M_{t_{i}}.$
\end{proof}
\end{proposition}

\begin{corollary}
Let $\gamma :\mathbb{R\rightarrow R}_{1}^{2}$, $\gamma \left( t\right)
=\left( \cosh t,\sinh t\right) $ and $\delta :\mathbb{R\rightarrow R}%
_{1}^{2} $ $\delta \left( s\right) =\left( \cosh s,\sinh s\right) $ be two
Lorentzian circles centered at O $\left( a,b\in \mathbb{R}\right) .$ Then
their tensor product is 2-dimensional Lie subgroup of $M_{t_{i}}^{\ast }.$
\end{corollary}

\begin{proposition}
Let $\gamma :\mathbb{R\rightarrow R}_{1}^{2}$, $\gamma \left( t\right)
=\left( \cosh t,\sinh t\right) $ and $\delta :\mathbb{R\rightarrow R}_{1}^{2}
$ $\delta \left( s\right) =\left( \cosh s,\sinh s\right) $ be two Lorentzian
circles centered at O $\left( a,b\in \mathbb{R}\right) .$ Then, the left
invariant vector fields on tensor product surface $f\left( t,s\right)
=\gamma \left( t\right) \otimes \delta \left( s\right) $ are $X_{2}$ and $%
X_{4}$ which are the left invariant vector fields on $M_{t_{i}}^{\ast }.$
\end{proposition}

\begin{remark}
The Case I in this subsection coincides the paper studied by Karaku\c{s}
\"{O}. and Yayl\i \ [4]. So it can be considered that this subsection is a
generalization of this study.
\end{remark}

\subsection{Tensor Product Surfaces on $M_{t_{j}}$ Hyperquadric and Some
Special Lie Subgroups}

In this subsection, we change the definition of tensor product as follows:

Let $\gamma :\mathbb{R\rightarrow R}_{k}^{2}$ $\left( +\text{ }-\alpha \beta
\right) $ and $\delta :\mathbb{R\rightarrow R}_{t}^{2}$ $\left( +\text{ }%
\beta \right) $ be planar curves in Euclidean or Lorentzian space. Put $%
\gamma \left( t\right) =\left( \gamma _{1}\left( t\right) ,\gamma _{2}\left(
t\right) \right) $ and $\delta \left( s\right) =\left( \delta _{1}\left(
s\right) ,\delta _{2}\left( s\right) \right) .$ Let us define their tensor
product as%
\begin{equation*}
f=\gamma \otimes \delta :\mathbb{R}^{2}\rightarrow \mathbb{R}_{v}^{4}\text{
\  \ }\left( +\text{ }\alpha \text{ }-\beta \text{ }-\alpha \beta \right) ,
\end{equation*}%
\begin{equation}
f\left( t,s\right) =\left( \gamma _{1}\left( t\right) \delta _{1}\left(
s\right) ,-\beta \gamma _{2}\left( t\right) \delta _{2}\left( s\right)
,\gamma _{1}\left( t\right) \delta _{2}\left( s\right) ,\gamma _{2}\left(
t\right) \delta _{1}\left( s\right) \right)
\end{equation}%
tensor product surface given by (3) is a surface on $M_{t_{j}}$
hyperquadric. Tangent vector fields of $f\left( t,s\right) $ can be easily
computed as%
\begin{eqnarray}
\frac{\partial f}{\partial t} &=&\left( \gamma _{1}^{\prime }\left( t\right)
\delta _{1}\left( s\right) ,-\beta \gamma _{2}^{\prime }\left( t\right)
\delta _{2}\left( s\right) ,\gamma _{1}^{\prime }\left( t\right) \delta
_{2}\left( s\right) ,\gamma _{2}^{\prime }\left( t\right) \delta _{1}\left(
s\right) \right)  \\
\frac{\partial f}{\partial s} &=&\left( \gamma _{1}\left( t\right) \delta
_{1}^{\prime }\left( s\right) ,-\beta \gamma _{2}\left( t\right) \delta
_{2}^{\prime }\left( s\right) ,\gamma _{1}\left( t\right) \delta
_{2}^{\prime }\left( s\right) ,\gamma _{2}\left( t\right) \delta
_{1}^{\prime }\left( s\right) \right)   \notag
\end{eqnarray}%
By using (4), we have
\begin{eqnarray*}
g_{11} &=&g\left( \frac{\partial f}{\partial t},\frac{\partial f}{\partial t}%
\right) =g_{1}(\gamma ^{\prime },\gamma ^{\prime })g_{2}\left( \delta
,\delta \right)  \\
g_{12} &=&g\left( \frac{\partial f}{\partial t},\frac{\partial f}{\partial s}%
\right) =g_{1}(\gamma ,\gamma ^{\prime })g_{2}\left( \delta ,\delta ^{\prime
}\right)  \\
g_{22} &=&g\left( \frac{\partial f}{\partial s},\frac{\partial f}{\partial s}%
\right) =g_{1}(\gamma ,\gamma )g_{2}\left( \delta ^{\prime },\delta ^{\prime
}\right)
\end{eqnarray*}%
where $g_{1}=dx_{1}^{2}-\alpha \beta dx_{2}^{2}$ and $g_{2}=dx_{1}^{2}+\beta
dx_{2}^{2}$ are the metrics of $\mathbb{R}_{k}^{2}$ and $\mathbb{R}_{t}^{2}$%
, respectively. Consequently, an orthonormal basis for the tangent space of $%
f(t,s)$ is given by%
\begin{eqnarray*}
e_{1} &=&\frac{1}{\sqrt{\left \vert g_{11}\right \vert }}\frac{\partial f}{%
\partial t} \\
e_{2} &=&\frac{1}{\sqrt{\left \vert g_{11}\left(
g_{11}g_{22}-g_{12}^{2}\right) \right \vert }}\left( g_{11}\frac{\partial f}{%
\partial s}-g_{12}\frac{\partial f}{\partial t}\right)
\end{eqnarray*}

\begin{remark}
Tensor product surface given by (3) is a surface in $\mathbb{R}^{4}$ or $%
\mathbb{R}_{2}^{4}$ according to the case of $\alpha $ and $\beta .$ If we
take as $\alpha =\beta =1,$ we obtain a tensor product surface of a
Lorentzian plane curve and a Euclidean plane curve in $\mathbb{R}_{2}^{4}.$
If we take as $\alpha =1$, $\beta =-1,$ we obtain a tensor product surface
of a Euclidean plane curve and a Lorentzian plane curve in $\mathbb{R}%
_{2}^{4}.$ If we take as $\alpha =-1$, $\beta =1,$ we obtain a tensor
product surface of two Euclidean plane curves in $\mathbb{R}^{4}.$ If we
take as $\alpha =-1$, $\beta =-1$ we obtain a tensor product surface of two
Lorentzian plane curves in $\mathbb{R}_{2}^{4}.$
\end{remark}

Now we investigate Lie group structure of tensor product surfaces given by
the parametrization (3) in $\mathbb{R}^{4}$ or $\mathbb{R}_{2}^{4}$
according to above cases. Morever we obtain left invariant vector fields of
the tensor product surface that has the structure of Lie group.

\subsubsection{Case I $\protect \alpha =\protect \beta =1$}

\begin{proposition}
Let $\gamma :\mathbb{R\rightarrow R}_{1}^{2}$ $\left( +\text{ }-\right) $ be
a hyperbolic spiral and $\delta :\mathbb{R\rightarrow R}^{2}$ $\left( +\text{
}+\right) $ be a spiral with the same parameter, i.e. $\gamma \left(
t\right) =e^{at}\left( \cosh t,\sinh t\right) $ and $\delta \left( t\right)
=e^{bt}\left( \cos t,\sin t\right) $ $\left( a,b\in \mathbb{R}\right) .$
Their tensor product is a one parameter subgroup of Lie group $M_{t_{j}}.$
\end{proposition}

\begin{proof}
We obtain
\begin{equation*}
\varphi \left( t\right) =\gamma \left( t\right) \otimes \delta \left(
t\right) =e^{(a+b)t}\left( \cosh t\cos t,-\sinh t\sin t,\cosh t\sin t,\sinh
t\cos t\right)
\end{equation*}%
It can be easily seen that
\begin{equation*}
\varphi \left( t_{1}\right) \cdot \varphi \left( t_{2}\right) =\varphi
\left( t_{1}+t_{2}\right)
\end{equation*}%
for all $t_{1},t_{2}.$ Also $\varphi ^{-1}\left( t\right) =\varphi \left(
-t\right) .$ Hence $\left( \varphi \left( t\right) ,\cdot \right) $ is a one
parameter Lie subgroup of Lie group $\left( M_{t_{j}},\cdot \right) .$
\end{proof}

\begin{corollary}
Let $\gamma :\mathbb{R\rightarrow R}_{1}^{2}$ be a Lorentzian circle
centered at O and $\delta :\mathbb{R\rightarrow R}^{2}$ be circle centered
at O with the same parameter, i.e., $\gamma \left( t\right) =\left( \cosh
t,\sinh t\right) $ and $\delta \left( t\right) =\left( \cos t,\sin t\right)
. $ Then their tensor product is a one parameter subgroup of Lie group $%
M_{t_{j}}^{\ast }.$
\end{corollary}

\begin{proposition}
Let $\varphi \left( t\right) $ be tensor product of a Lorentzian cirle
centered at O and a circle centered at O with the same parameter. Then the
left invariant vector field on $\varphi \left( t\right) $ is $X=X_{3}+X_{4},$
where $X_{3}$ and $X_{4}$ are left invariant vector fields on $%
M_{t_{j}}^{\ast }.$
\end{proposition}

\begin{proposition}
Let $\gamma :\mathbb{R\rightarrow R}_{1}^{2}$, $\gamma \left( t\right)
=e^{at}\left( \cosh t,\sinh t\right) $ be a hyperbolic spiral and $\delta :%
\mathbb{R\rightarrow R}^{2}$ $\delta \left( s\right) =e^{bs}\left( \cos
s,\sin s\right) $ be a spiral $\left( a,b\in \mathbb{R}\right) .$ Then their
tensor product is 2-dimensional Lie subgroup of $M_{t_{i}}.$

\begin{proof}
By using tensor product rule given by (3), we get%
\begin{equation*}
f\left( t,s\right) =e^{at+bs}\left( \cosh t\cos s,-\sinh t\sin s,\cosh t\sin
s,\sinh t\cos s\right)
\end{equation*}%
Every point of $f\left( t,s\right) $\ is on the hyperquadric $M_{t_{j}}.$
Since $f\left( t,s\right) $ is both subgroup and submanifold of a Lie group $%
M_{t_{j}},$ it is a 2-dimensional Lie subgroup of $M_{t_{j}}.$
\end{proof}
\end{proposition}

\begin{proposition}
Let $\gamma :\mathbb{R\rightarrow R}_{1}^{2}$, $\gamma \left( t\right)
=\left( \cosh t,\sinh t\right) $ be a Lorentzian circle and $\delta :\mathbb{%
R\rightarrow R}^{2}$ $\delta \left( s\right) =\left( \cos s,\sin s\right) $
be a circle $\left( a,b\in \mathbb{R}\right) .$ Then their tensor product is
2-dimensional Lie subgroup of $M_{t_{j}}^{\ast }.$
\end{proposition}

\begin{proposition}
Let $\gamma :\mathbb{R\rightarrow R}_{1}^{2}$, $\gamma \left( t\right)
=\left( \cosh t,\sinh t\right) $ be a Lorentzian circle at centered O and $%
\delta :\mathbb{R\rightarrow R}^{2}$ $\delta \left( s\right) =\left( \cos
s,\sin s\right) $ be a circle at centered O $\left( a,b\in \mathbb{R}\right)
.$ Then, the left invariant vector fields on tensor product surface $f\left(
t,s\right) =\gamma \left( t\right) \otimes \delta \left( s\right) $ are $%
X_{3}$ and $X_{4}$ which are the left invariant vector fields on $%
M_{t_{j}}^{\ast }.$
\end{proposition}

\subsubsection{Case II $\protect \alpha =-1,\protect \beta =1$}

\begin{proposition}
Let $\gamma :\mathbb{R\rightarrow R}^{2}$ $\left( +\text{ }+\right) $ and $%
\delta :\mathbb{R\rightarrow R}^{2}$ $\left( +\text{ }+\right) $ be two
spirals with the same parameter, i.e. $\gamma \left( t\right) =e^{at}\left(
\cos t,\sin t\right) $ and $\delta \left( t\right) =e^{bt}\left( \cos t,\sin
t\right) $ $\left( a,b\in \mathbb{R}\right) .$ Then their tensor product is
a one parameter subgroup of Lie group $M_{t_{i}}.$
\end{proposition}

\begin{proof}
We obtain
\begin{equation*}
\varphi \left( t\right) =\gamma \left( t\right) \otimes \delta \left(
t\right) =e^{(a+b)t}\left( \cos ^{2}t,-\sin ^{2}t,\cos t\sin t,\cos t\sin
t\right)
\end{equation*}%
It can be easily seen that
\begin{equation*}
\varphi \left( t_{1}\right) \cdot \varphi \left( t_{2}\right) =\varphi
\left( t_{1}+t_{2}\right)
\end{equation*}%
for all $t_{1},t_{2}.$ Also $\varphi ^{-1}\left( t\right) =\varphi \left(
-t\right) .$ Hence $\left( \varphi \left( t\right) ,\cdot \right) $ is a one
parameter Lie subgroup of Lie group $\left( M_{t_{j}},\cdot \right) .$
\end{proof}

\begin{corollary}
Let $\gamma :\mathbb{R\rightarrow R}^{2}$ and $\delta :\mathbb{R\rightarrow R%
}^{2}$ be two circles centered at O with the same parameter, i.e., $\gamma
\left( t\right) =\left( \cos t,\sin t\right) $ and $\delta \left( t\right)
=\left( \cos t,\sin t\right) .$ Then their tensor product is a one parameter
subgroup of Lie group $M_{t_{j}}^{\ast }.$
\end{corollary}

\begin{proposition}
Let $\varphi \left( t\right) $ be tensor product of two circles centered at
O with the same parameter. Then the left invariant vector field on $\varphi
\left( t\right) $ is $X=X_{3}+X_{4},$ where $X_{3}$ and $X_{4}$ are left
invariant vector fields on $M_{t_{j}}^{\ast }.$
\end{proposition}

\begin{proposition}
Let $\gamma :\mathbb{R\rightarrow R}^{2}$, $\gamma \left( t\right)
=e^{at}\left( \cos t,\sin t\right) $ and $\delta :\mathbb{R\rightarrow R}%
^{2} $ $\delta \left( s\right) =e^{bs}\left( \cos s,\sin s\right) $ be two
spirals $\left( a,b\in \mathbb{R}\right) .$ Then their tensor product is
2-dimensional Lie subgroup of $M_{t_{j}}.$

\begin{proof}
By using tensor product rule given by (3), we get%
\begin{equation*}
f\left( t,s\right) =e^{at+bs}\left( \cos t\cos s,-\sin t\sin s,\cos t\sin
s,,\sin t\cos s\right)
\end{equation*}%
Every point of $f\left( t,s\right) $\ is on the hyperquadric $M_{t_{j}}.$
Since $f\left( t,s\right) $ is both subgroup and submanifold of a Lie group $%
M_{t_{j}},$ it is a 2-dimensional Lie subgroup of $M_{t_{j}}.$
\end{proof}
\end{proposition}

\begin{corollary}
Let $\gamma :\mathbb{R\rightarrow R}^{2}$, $\gamma \left( t\right) =\left(
\cos t,\sin t\right) $ and $\delta :\mathbb{R\rightarrow R}^{2}$ $\delta
\left( s\right) =\left( \cos s,\sin s\right) $ be two circles centered at O $%
\left( a,b\in \mathbb{R}\right) .$ Then their tensor product is
2-dimensional Lie subgroup of $M_{t_{j}}^{\ast }.$
\end{corollary}

\begin{proposition}
Let $\gamma :\mathbb{R\rightarrow R}_{1}^{2}$, $\gamma \left( t\right)
=\left( \cos t,\sin t\right) $ and $\delta :\mathbb{R\rightarrow R}^{2}$ $%
\delta \left( s\right) =\left( \cos s,\sin s\right) $ be two circles
centered at O $\left( a,b\in \mathbb{R}\right) .$ Then, the left invariant
vector fields on tensor product surface $f\left( t,s\right) =\gamma \left(
t\right) \otimes \delta \left( s\right) $ are $X_{3}$ and $X_{4}$ which are
the left invariant vector fields on $M_{t_{j}}^{\ast }.$
\end{proposition}

\subsubsection{Case III $\protect \alpha =-1,\protect \beta =-1$}

\begin{proposition}
Let $\gamma :\mathbb{R\rightarrow R}_{1}^{2}$ $\left( +\text{ }-\right) $
and $\delta :\mathbb{R\rightarrow R}_{1}^{2}$ $\left( +\text{ }-\right) $ be
two hyperbolic spirals with the same parameter, i.e. $\gamma \left( t\right)
=e^{at}\left( \cosh t,\sinh t\right) $ and $\delta \left( t\right)
=e^{bt}\left( \cosh t,\sinh t\right) $ $\left( a,b\in \mathbb{R}\right) .$
Then their tensor product is a one parameter subgroup of Lie group $%
M_{t_{j}}.$
\end{proposition}

\begin{proof}
We obtain
\begin{equation*}
\varphi \left( t\right) =\gamma \left( t\right) \otimes \delta \left(
t\right) =e^{(a+b)t}\left( \cosh ^{2}t,\sinh ^{2}t,\cosh t\sinh t,\cosh
t,\sinh t\right)
\end{equation*}%
It can be easily seen that
\begin{equation*}
\varphi \left( t_{1}\right) \cdot \varphi \left( t_{2}\right) =\varphi
\left( t_{1}+t_{2}\right)
\end{equation*}%
for all $t_{1},t_{2}.$ Also $\varphi ^{-1}\left( t\right) =\varphi \left(
-t\right) .$ Hence $\left( \varphi \left( t\right) ,\cdot \right) $ is a one
parameter Lie subgroup of Lie group $\left( M_{t_{j}},\cdot \right) .$
\end{proof}

\begin{corollary}
Let $\gamma :\mathbb{R\rightarrow R}_{1}^{2}$ and $\delta :\mathbb{%
R\rightarrow R}_{1}^{2}$ be two Lorentzian circles centered at O with the
same parameter, i.e., $\gamma \left( t\right) =\left( \cosh t,\sinh t\right)
$ and $\delta \left( t\right) =\left( \cosh t,\sinh t\right) .$ Then their
tensor product is a one parameter subgroup of Lie group $M_{t_{j}}^{\ast }.$
\end{corollary}

\begin{proposition}
Let $\varphi \left( t\right) $ be tensor product of two Lorentzian circles
centered at O with the same parameter. Then the left invariant vector field
on $\varphi \left( t\right) $ is $X=X_{3}+X_{4},$ where $X_{3}$ and $X_{4}$
are left invariant vector fields on $M_{t_{j}}^{\ast }.$
\end{proposition}

\begin{proposition}
Let $\gamma :\mathbb{R\rightarrow R}_{1}^{2}$, $\gamma \left( t\right)
=e^{at}\left( \cosh t,\sinh t\right) $ and $\delta :\mathbb{R\rightarrow R}%
_{1}^{2}$ $\delta \left( s\right) =e^{bs}\left( \cosh s,\sinh s\right) $ be
two spirals $\left( a,b\in \mathbb{R}\right) .$ Then their tensor product is
2-dimensional Lie subgroup of $M_{t_{j}}.$

\begin{proof}
By using tensor product rule given by (3), we get%
\begin{equation*}
f\left( t,s\right) =e^{at+bs}\left( \cosh t\cosh s,\sinh t\sinh s,\cosh
t\sinh s,\sinh t\cosh s\right)
\end{equation*}%
Every point of $f\left( t,s\right) $\ is on the hyperquadric $M_{t_{j}}.$
Since $f\left( t,s\right) $ is both subgroup and submanifold of a Lie group $%
M_{t_{j}},$ it is a 2-dimensional Lie subgroup of $M_{t_{j}}.$
\end{proof}
\end{proposition}

\begin{corollary}
Let $\gamma :\mathbb{R\rightarrow R}_{1}^{2}$, $\gamma \left( t\right)
=\left( \cosh t,\sinh t\right) $ and $\delta :\mathbb{R\rightarrow R}%
_{1}^{2} $ $\delta \left( s\right) =\left( \cosh s,\sinh s\right) $ be two
Lorentzian circles centered at O $\left( a,b\in \mathbb{R}\right) .$ Then
their tensor product is 2-dimensional Lie subgroup of $M_{t_{j}}^{\ast }.$
\end{corollary}

\begin{proposition}
Let $\gamma :\mathbb{R\rightarrow R}_{1}^{2}$, $\gamma \left( t\right)
=\left( \cosh t,\sinh t\right) $ and $\delta :\mathbb{R\rightarrow R}%
_{1}^{2} $ $\delta \left( s\right) =\left( \cosh s,\sinh s\right) $ be two
Lorentzian circles centered at O $\left( a,b\in \mathbb{R}\right) .$ Then,
the left invariant vector fields on tensor product surface $f\left(
t,s\right) =\gamma \left( t\right) \otimes \delta \left( s\right) $ are $%
X_{3}$ and $X_{4}$ which are the left invariant vector fields on $%
M_{t_{j}}^{\ast }.$
\end{proposition}

\subsection{Tensor Product Surfaces on $M_{t_{ij}}$ Hyperquadric and Some
Special Lie Subgroups}

In this subsection, we use the definition of tensor product given by Mihai.

Let $\gamma :\mathbb{R\rightarrow R}_{k}^{2}$ $\left( +\text{ }\beta \right)
$ and $\delta :\mathbb{R\rightarrow R}_{t}^{2}$ $\left( +\text{ }\alpha
\right) $ be planar curves in Euclidean or Lorentzian space. Put $\gamma
\left( t\right) =\left( \gamma _{1}\left( t\right) ,\gamma _{2}\left(
t\right) \right) $ and $\delta \left( s\right) =\left( \delta _{1}\left(
s\right) ,\delta _{2}\left( s\right) \right) .$%
\begin{equation*}
f=\gamma \otimes \delta :\mathbb{R}^{2}\rightarrow \mathbb{R}_{v}^{4}\text{
\  \ }\left( +\text{ }\alpha \text{ }\beta \text{ }\alpha \beta \right) ,
\end{equation*}%
\begin{equation}
f\left( t,s\right) =\left( \gamma _{1}\left( t\right) \delta _{1}\left(
s\right) ,\gamma _{1}\left( t\right) \delta _{2}\left( s\right) ,\gamma
_{2}\left( t\right) \delta _{1}\left( s\right) ,\gamma _{2}\left( t\right)
\delta _{2}\left( s\right) \right)
\end{equation}%
tensor product surface given by (5) is a surface on $M_{t_{ij}}$
hyperquadric. Tangent vector fields of $f\left( t,s\right) $ can be easily
computed as%
\begin{eqnarray}
\frac{\partial f}{\partial t} &=&\left( \gamma _{1}^{\prime }\left( t\right)
\delta _{1}\left( s\right) ,\gamma _{1}^{\prime }\left( t\right) \delta
_{2}\left( s\right) ,\gamma _{2}^{\prime }\left( t\right) \delta _{1}\left(
s\right) ,\gamma _{2}^{\prime }\left( t\right) \delta _{2}\left( s\right)
\right) \\
\frac{\partial f}{\partial s} &=&\left( \gamma _{1}\left( t\right) \delta
_{1}^{\prime }\left( s\right) ,\gamma _{1}\left( t\right) \delta
_{2}^{\prime }\left( s\right) ,\gamma _{2}\left( t\right) \delta
_{1}^{\prime }\left( s\right) ,\gamma _{2}\left( t\right) \delta
_{2}^{\prime }\left( s\right) \right)  \notag
\end{eqnarray}%
By using (5) and (6), we have
\begin{eqnarray*}
g_{11} &=&g\left( \frac{\partial f}{\partial t},\frac{\partial f}{\partial t}%
\right) =g_{1}(\gamma ^{\prime },\gamma ^{\prime })g_{2}\left( \delta
,\delta \right) \\
g_{12} &=&g\left( \frac{\partial f}{\partial t},\frac{\partial f}{\partial s}%
\right) =g_{1}(\gamma ,\gamma ^{\prime })g_{2}\left( \delta ,\delta ^{\prime
}\right) \\
g_{22} &=&g\left( \frac{\partial f}{\partial s},\frac{\partial f}{\partial s}%
\right) =g_{1}(\gamma ,\gamma )g_{2}\left( \delta ^{\prime },\delta ^{\prime
}\right)
\end{eqnarray*}%
where $g_{1}=dx_{1}^{2}+\beta dx_{2}^{2}$ and $g_{2}=dx_{1}^{2}+\alpha
dx_{2}^{2}$ are the metrics of $\mathbb{R}_{k}^{2}$ and $\mathbb{R}_{t}^{2}$%
, respectively. Consequently, an orthonormal basis for the tangent space of $%
f(t,s)$ is given by%
\begin{eqnarray*}
e_{1} &=&\frac{1}{\sqrt{\left \vert g_{11}\right \vert }}\frac{\partial f}{%
\partial t} \\
e_{2} &=&\frac{1}{\sqrt{\left \vert g_{11}\left(
g_{11}g_{22}-g_{12}^{2}\right) \right \vert }}\left( g_{11}\frac{\partial f}{%
\partial s}-g_{12}\frac{\partial f}{\partial t}\right)
\end{eqnarray*}

\begin{remark}
Tensor product surface given by (5) is a surface in $\mathbb{R}^{4}$ or $%
\mathbb{R}_{2}^{4}$ according to the case of $\alpha $ and $\beta .$ If we
take as $\alpha =\beta =1,$ we obtain a tensor product surface of two
Euclidean plane curves in $\mathbb{R}^{4}.$ If we take as $\alpha =1$, $%
\beta =-1,$ we obtain a tensor product surface of a Lorentzian plane curve
and a Euclidean plane curve in $\mathbb{R}_{2}^{4}.$ If we take as $\alpha
=-1$, $\beta =1,$ we obtain a tensor product surface of a Euclidean plane
curve and a Lorentzian plane curve in $\mathbb{R}_{2}^{4}.$ If we take as $%
\alpha =-1$, $\beta =-1$ we obtain a tensor product surface of two
Lorentzian plane curves in $\mathbb{R}_{2}^{4}.$
\end{remark}

Now we investigate Lie group structure of tensor product surfaces given by
the parametrization (5) in $\mathbb{R}^{4}$ or $\mathbb{R}_{2}^{4}$
according to above cases. Morever we obtain left invariant vector fields of
the tensor product surface that has the structure of Lie group.

\subsubsection{Case I $\protect \alpha =\protect \beta =1$}

\begin{proposition}
Let $\gamma :\mathbb{R\rightarrow R}^{2}$ $\left( +\text{ }+\right) $ and $%
\delta :\mathbb{R\rightarrow R}^{2}$ $\left( +\text{ }+\right) $ be two
spirals with the same parameter, i.e. $\gamma \left( t\right) =e^{at}\left(
\cos t,\sin t\right) $ and $\delta \left( t\right) =e^{bt}\left( \cos t,\sin
t\right) $ $\left( a,b\in \mathbb{R}\right) .$ Then their tensor product is
a one parameter subgroup of Lie group $M_{t_{ij}}.$
\end{proposition}

\begin{proof}
We obtain
\begin{equation*}
\varphi \left( t\right) =\gamma \left( t\right) \otimes \delta \left(
t\right) =e^{(a+b)t}\left( \cos ^{2}t,\cos t\sin t,\cos t\sin t,\sin
^{2}t\right)
\end{equation*}%
It can be easily seen that
\begin{equation*}
\varphi \left( t_{1}\right) \cdot \varphi \left( t_{2}\right) =\varphi
\left( t_{1}+t_{2}\right)
\end{equation*}%
for all $t_{1},t_{2}.$ Also $\varphi ^{-1}\left( t\right) =\varphi \left(
-t\right) .$ Hence $\left( \varphi \left( t\right) ,\cdot \right) $ is a one
parameter Lie subgroup of Lie group $\left( M_{t_{ij}},\cdot \right) .$
\end{proof}

\begin{corollary}
Let $\gamma :\mathbb{R\rightarrow R}^{2}$ and $\delta :\mathbb{R\rightarrow R%
}^{2}$ be two circles centered at O with the same parameter, i.e., $\gamma
\left( t\right) =\left( \cos t,\sin t\right) $ and $\delta \left( t\right)
=\left( \cos t,\sin t\right) .$ Then their tensor product is a one parameter
subgroup of Lie group $M_{t_{ij}}^{\ast }.$
\end{corollary}

\begin{proposition}
Let $\varphi \left( t\right) $ be tensor product of two circles centered at
O with the same parameter. Then the left invariant vector field on $\varphi
\left( t\right) $ is $X=X_{2}+X_{3},$ where $X_{2}$ and $X_{3}$ are left
invariant vector fields on $M_{t_{ij}}^{\ast }.$
\end{proposition}

\begin{proposition}
Let $\gamma :\mathbb{R\rightarrow R}^{2}$, $\gamma \left( t\right)
=e^{at}\left( \cos t,\sin t\right) $ and $\delta :\mathbb{R\rightarrow R}%
^{2} $ $\delta \left( s\right) =e^{bs}\left( \cos s,\sin s\right) $ be two
spirals $\left( a,b\in \mathbb{R}\right) .$ Then their tensor product is
2-dimensional Lie subgroup of $M_{t_{ij}}.$

\begin{proof}
By using tensor product rule given by (5), we get%
\begin{equation*}
f\left( t,s\right) =e^{at+bs}\left( \cos t\cos s,\cos t\sin s,\sin t\cos
s,\sin t\sin s\right)
\end{equation*}%
Every point of $f\left( t,s\right) $\ is on the hyperquadric $M_{t_{ij}}.$
Since $f\left( t,s\right) $ is both subgroup and submanifold of a Lie group $%
M_{t_{ij}},$ it is a 2-dimensional Lie subgroup of $M_{t_{ij}}.$
\end{proof}
\end{proposition}

\begin{corollary}
Let $\gamma :\mathbb{R\rightarrow R}^{2}$, $\gamma \left( t\right) =\left(
\cos t,\sin t\right) $ and $\delta :\mathbb{R\rightarrow R}^{2}$ $\delta
\left( s\right) =\left( \cos s,\sin s\right) $ be two circles centered at O $%
\left( a,b\in \mathbb{R}\right) .$ Then their tensor product is
2-dimensional Lie subgroup of $M_{t_{ij}}^{\ast }.$
\end{corollary}

\begin{proposition}
Let $\gamma :\mathbb{R\rightarrow R}_{1}^{2}$, $\gamma \left( t\right)
=\left( \cos t,\sin t\right) $ and $\delta :\mathbb{R\rightarrow R}^{2}$ $%
\delta \left( s\right) =\left( \cos s,\sin s\right) $ be two circles
centered at O $\left( a,b\in \mathbb{R}\right) .$ Then, the left invariant
vector fields on tensor product surface $f\left( t,s\right) =\gamma \left(
t\right) \otimes \delta \left( s\right) $ are $X_{2}$ and $X_{3}$ which are
the left invariant vector fields on $M_{t_{ij}}^{\ast }.$
\end{proposition}

\subsubsection{Case II $\protect \alpha =1,\protect \beta =-1$}

\begin{proposition}
Let $\gamma :\mathbb{R\rightarrow R}_{1}^{2}$ $\left( +\text{ }-\right) $ be
a hyperbolic spiral and $\delta :\mathbb{R\rightarrow R}^{2}$ $\left( +\text{
}+\right) $ be a spiral with the same parameter, i.e. $\gamma \left(
t\right) =e^{at}\left( \cosh t,\sinh t\right) $ and $\delta \left( t\right)
=e^{bt}\left( \cos t,\sin t\right) $ $\left( a,b\in \mathbb{R}\right) .$
Their tensor product is a one parameter subgroup of Lie group $M_{t_{ij}}.$
\end{proposition}

\begin{proof}
We obtain
\begin{equation*}
\varphi \left( t\right) =\gamma \left( t\right) \otimes \delta \left(
t\right) =e^{(a+b)t}\left( \cosh t\cos t,\cosh t\sin t,\sinh t\cos t,\sinh
t\sin t\right)
\end{equation*}%
It can be easily seen that
\begin{equation*}
\varphi \left( t_{1}\right) \cdot \varphi \left( t_{2}\right) =\varphi
\left( t_{1}+t_{2}\right)
\end{equation*}%
for all $t_{1},t_{2}.$ Also $\varphi ^{-1}\left( t\right) =\varphi \left(
-t\right) .$ Hence $\left( \varphi \left( t\right) ,\cdot \right) $ is a one
parameter Lie subgroup of Lie group $\left( M_{t_{ij}},\cdot \right) .$
\end{proof}

\begin{corollary}
Let $\gamma :\mathbb{R\rightarrow R}_{1}^{2}$ be a Lorentzian circle
centered at O and $\delta :\mathbb{R\rightarrow R}^{2}$ be circle centered
at O with the same parameter, i.e., $\gamma \left( t\right) =\left( \cosh
t,\sinh t\right) $ and $\delta \left( t\right) =\left( \cos t,\sin t\right)
. $ Then their tensor product is a one parameter subgroup of Lie group $%
M_{t_{ij}}^{\ast }.$
\end{corollary}

\begin{proposition}
Let $\varphi \left( t\right) $ be tensor product of a Lorentzian cirle
centered at O and a circle centered at O with the same parameter. Then the
left invariant vector field on $\varphi \left( t\right) $ is $X=X_{2}+X_{3},$
where $X_{2}$ and $X_{3}$ are left invariant vector fields on $%
M_{t_{ij}}^{\ast }.$
\end{proposition}

\begin{proposition}
Let $\gamma :\mathbb{R\rightarrow R}_{1}^{2}$, $\gamma \left( t\right)
=e^{at}\left( \cosh t,\sinh t\right) $ be a hyperbolic spiral and $\delta :%
\mathbb{R\rightarrow R}^{2}$ $\delta \left( s\right) =e^{bs}\left( \cos
s,\sin s\right) $ be a spiral $\left( a,b\in \mathbb{R}\right) .$ Then their
tensor product is 2-dimensional Lie subgroup of $M_{t_{i}}.$

\begin{proof}
By using tensor product rule given by (5), we get%
\begin{equation*}
f\left( t,s\right) =e^{at+bs}\left( \cosh t\cos s,\cosh t\sin s,\sinh t\cos
s,\sinh t\sin s\right)
\end{equation*}%
Every point of $f\left( t,s\right) $\ is on the hyperquadric $M_{t_{ij}}.$
Since $f\left( t,s\right) $ is both subgroup and submanifold of a Lie group $%
M_{t_{j}},$ it is a 2-dimensional Lie subgroup of $M_{t_{ij}}.$
\end{proof}
\end{proposition}

\begin{proposition}
Let $\gamma :\mathbb{R\rightarrow R}_{1}^{2}$, $\gamma \left( t\right)
=\left( \cosh t,\sinh t\right) $ be a Lorentzian circle and $\delta :\mathbb{%
R\rightarrow R}^{2}$ $\delta \left( s\right) =\left( \cos s,\sin s\right) $
be a circle $\left( a,b\in \mathbb{R}\right) .$ Then their tensor product is
2-dimensional Lie subgroup of $M_{t_{ij}}^{\ast }.$
\end{proposition}

\begin{proposition}
Let $\gamma :\mathbb{R\rightarrow R}_{1}^{2}$, $\gamma \left( t\right)
=\left( \cosh t,\sinh t\right) $ be a Lorentzian circle at centered O and $%
\delta :\mathbb{R\rightarrow R}^{2}$ $\delta \left( s\right) =\left( \cos
s,\sin s\right) $ be a circle at centered O $\left( a,b\in \mathbb{R}\right)
.$ Then, the left invariant vector fields on tensor product surface $f\left(
t,s\right) =\gamma \left( t\right) \otimes \delta \left( s\right) $ are $%
X_{2}$ and $X_{3}$ which are the left invariant vector fields on $%
M_{t_{ij}}^{\ast }.$
\end{proposition}

\subsubsection{Case III $\protect \alpha =-1,\protect \beta =-1$}

\begin{proposition}
Let $\gamma :\mathbb{R\rightarrow R}_{1}^{2}$ $\left( +\text{ }-\right) $
and $\delta :\mathbb{R\rightarrow R}_{1}^{2}$ $\left( +\text{ }-\right) $ be
two hyperbolic spirals with the same parameter, i.e. $\gamma \left( t\right)
=e^{at}\left( \cosh t,\sinh t\right) $ and $\delta \left( t\right)
=e^{bt}\left( \cosh t,\sinh t\right) $ $\left( a,b\in \mathbb{R}\right) .$
Then their tensor product is a one parameter subgroup of Lie group $%
M_{t_{ij}}.$
\end{proposition}

\begin{proof}
We obtain
\begin{equation*}
\varphi \left( t\right) =\gamma \left( t\right) \otimes \delta \left(
t\right) =e^{(a+b)t}\left( \cosh ^{2}t,\cosh t\sinh t,\cosh t,\sinh t,\sinh
^{2}t\right)
\end{equation*}%
It can be easily seen that
\begin{equation*}
\varphi \left( t_{1}\right) \cdot \varphi \left( t_{2}\right) =\varphi
\left( t_{1}+t_{2}\right)
\end{equation*}%
for all $t_{1},t_{2}.$ Also $\varphi ^{-1}\left( t\right) =\varphi \left(
-t\right) .$ Hence $\left( \varphi \left( t\right) ,\cdot \right) $ is a one
parameter Lie subgroup of Lie group $\left( M_{t_{ij}},\cdot \right) .$
\end{proof}

\begin{corollary}
Let $\gamma :\mathbb{R\rightarrow R}_{1}^{2}$ and $\delta :\mathbb{%
R\rightarrow R}_{1}^{2}$ be two Lorentzian circles centered at O with the
same parameter, i.e., $\gamma \left( t\right) =\left( \cosh t,\sinh t\right)
$ and $\delta \left( t\right) =\left( \cosh t,\sinh t\right) .$ Then their
tensor product is a one parameter subgroup of Lie group $M_{t_{ij}}^{\ast }.$
\end{corollary}

\begin{proposition}
Let $\varphi \left( t\right) $ be tensor product of two Lorentzian circles
centered at O with the same parameter. Then the left invariant vector field
on $\varphi \left( t\right) $ is $X=X_{2}+X_{3},$ where $X_{2}$ and $X_{3}$
are left invariant vector fields on $M_{t_{ij}}^{\ast }.$
\end{proposition}

\begin{proposition}
Let $\gamma :\mathbb{R\rightarrow R}_{1}^{2}$, $\gamma \left( t\right)
=e^{at}\left( \cosh t,\sinh t\right) $ and $\delta :\mathbb{R\rightarrow R}%
_{1}^{2}$ $\delta \left( s\right) =e^{bs}\left( \cosh s,\sinh s\right) $ be
two spirals $\left( a,b\in \mathbb{R}\right) .$ Then their tensor product is
2-dimensional Lie subgroup of $M_{t_{ij}}.$

\begin{proof}
By using tensor product rule given by (5), we get%
\begin{equation*}
f\left( t,s\right) =e^{at+bs}\left( \cosh t\cosh s,\cosh t\sinh s,\sinh
t\cosh s,,\sinh t\sinh s\right)
\end{equation*}%
Every point of $f\left( t,s\right) $\ is on the hyperquadric $M_{t_{ij}}.$
Since $f\left( t,s\right) $ is both subgroup and submanifold of a Lie group $%
M_{t_{ij}},$ it is a 2-dimensional Lie subgroup of $M_{t_{ij}}.$
\end{proof}
\end{proposition}

\begin{corollary}
Let $\gamma :\mathbb{R\rightarrow R}_{1}^{2}$, $\gamma \left( t\right)
=\left( \cosh t,\sinh t\right) $ and $\delta :\mathbb{R\rightarrow R}%
_{1}^{2} $ $\delta \left( s\right) =\left( \cosh s,\sinh s\right) $ be two
Lorentzian circles centered at O $\left( a,b\in \mathbb{R}\right) .$ Then
their tensor product is 2-dimensional Lie subgroup of $M_{t_{ij}}^{\ast }.$
\end{corollary}

\begin{proposition}
Let $\gamma :\mathbb{R\rightarrow R}_{1}^{2}$, $\gamma \left( t\right)
=\left( \cosh t,\sinh t\right) $ and $\delta :\mathbb{R\rightarrow R}_{1}^{2}
$ $\delta \left( s\right) =\left( \cosh s,\sinh s\right) $ be two Lorentzian
circles centered at O $\left( a,b\in \mathbb{R}\right) .$ Then, the left
invariant vector fields on tensor product surface $f\left( t,s\right)
=\gamma \left( t\right) \otimes \delta \left( s\right) $ are $X_{2}$ and $%
X_{3}$ which are the left invariant vector fields on $M_{t_{ij}}^{\ast }.$
\end{proposition}

\begin{remark}
The Case I in this subsection 5.3 coincides the paper studied by \"{O}%
zkald\i \ and Yayl\i \ [8]. So it can be considered that this subsection is a
generalization of this study.
\end{remark}

\end{document}